\documentclass[a4paper,12pt]{article}
\usepackage{amsmath,amsfonts,amssymb,amsthm,enumerate,fullpage}
\usepackage{hyperref,booktabs}
\usepackage[utf8]{inputenc}

\providecommand{\EE}{\mathbb{E}}

\DeclareMathOperator{\Inf}{Inf}

\newtheorem{theorem}{Theorem}[section]
\newtheorem{lemma}[theorem]{Lemma}
\newtheorem{corollary}[theorem]{Corollary}

\theoremstyle{definition}

\title{Boolean constant degree functions on the slice \\ are juntas}
\author{Yuval Filmus\footnote{\hbadness=1500 Technion --- Israel Institute of Technology, Haifa, Israel.
This research was funded by ISF grant 1337/16.
},
Ferdinand Ihringer\footnote{\hbadness=1500 Einstein Institute of Mathematics, Hebrew University of Jerusalem, Israel.
Department of Pure Mathematics and Computer Algebra, Ghent University, Belgium. Supported by ERC advanced grant 320924.
The author is supported by a postdoctoral fellowship of the Research Foundation --- Flanders (FWO).}
}
\begin{document}

\maketitle

\begin{abstract}
 We show that a Boolean degree~$d$ function on the slice $\binom{[n]}{k} = \{ (x_1,\ldots,x_n) \in \{0,1\} : \sum_{i=1}^n x_i = k \}$ is a junta, assuming that $k,n-k$ are large enough. This generalizes a classical result of Nisan and Szegedy on the hypercube.
 Moreover, we show that the maximum number of coordinates that a Boolean degree~$d$ function can depend on is the same on the slice and the hypercube.
\end{abstract}

\section{Introduction} \label{sec:intro}

Nisan and Szegedy~\cite{NisanSzegedy} showed that a Boolean degree~$d$ function on the hypercube $\{0,1\}^n$ depends on at most $d2^{d-1}$ coordinates, and described a Boolean degree~$d$ function which depends on $\Omega(2^d)$ coordinates. Let us denote the optimal bound by $\gamma(d)$. The goal of this paper is to generalize this result to the \emph{slice} $\binom{[n]}{k}$ or \emph{Johnson scheme} $J(n,k)$, which consists of all points in the hypercube having Hamming weight~$k$:

\begin{theorem} \label{thm:main}
 There exists a constant $C$ such that the following holds. If $C^d \leq k \leq n-C^d$ and $f\colon \binom{[n]}{k} \to \{0,1\}$ has degree~$d$, then $f$ depends on at most $\gamma(d)$ coordinates.
\end{theorem}
(We explain in Section~\ref{sec:prel} what degree~$d$ means for functions on the hypercube and on the slice.)

Filmus et al.~\cite{fkmw} proved a version of Theorem~\ref{thm:main} (with a non-optimal bound on the number of points) when $k/n$ is bounded away from $0,1$, but their bound deteriorates as $k/n$ gets closer to~$0,1$. We use their result (which we reproduce here, to keep the proof self-contained) to bootstrap our own inductive argument.

The case $d = 1$ is much easier. The following folklore result is proved formally in~\cite{FilmusIhringer1}:

\begin{theorem} \label{thm:degree1}
 If $2 \leq k \leq n-2$ and $f\colon \binom{[n]}{k} \to \{0,1\}$ has degree~$1$, then $f$ depends on at most one coordinate.
\end{theorem}

The bounds on~$k$ in this theorem are optimal, since every function on $\binom{[n]}{1}$ and on $\binom{[n]}{n-1}$ has degree~$1$. In contrast, the bounds on~$k$ in Theorem~\ref{thm:main} are probably not optimal, an issue we discuss in Section~\ref{sec:discussion}.

Let us close this introduction by mentioning a recent result of Keller and Klein~\cite{KellerKlein}, which studies Boolean functions on $\binom{[n]}{k}$ which are $\epsilon$-close to being degree~$d$, where distance is measured using the squared $L_2$~norm. Assuming that $k \leq n/2$, their result states that if $\epsilon < (k/n)^{O(d)}$ then $f$ is $O(\epsilon)$-close to a junta.

\section{Preliminaries} \label{sec:prel}

In this paper, we discuss Boolean functions, which are $0,1$-valued functions, on two different domains: the hypercube and the slice.
We will use the notation $[n] := \{1,\ldots,n\}$. A \emph{degree~$d$ function} (in a context in which degree is defined) is a function of degree \emph{at most}~$d$. 

\paragraph{The hypercube.} The $n$-dimensional hypercube is the domain $\{0,1\}^n$. Every function on the hypercube can be represented uniquely as a multilinear polynomial in the $n$~input arguments $x_1,\ldots,x_n$. The \emph{degree} of a function on the hypercube is the degree of this polynomial. Alternatively, the degree of a function on the hypercube is the minimum degree of a polynomial in $x_1,\ldots,x_n$ which agrees with the function on all points of the hypercube. A function on the hypercube is an \emph{$m$-junta} if it depends on at most $m$~inputs, that is, if there exists a set $I$ of $m$~inputs such that $f(x) = f(y)$ as long as $x_i = y_i$ for all $i \in I$; we also say that $f$ is an $I$-junta. For more information on functions on the hypercube from this perspective, consult O'Donnell's monograph~\cite{RyanODonnell}.

\paragraph{The slice.} Let $0 \leq k \leq n$. The \emph{slice} $\binom{[n]}{k}$ is the subset of $\{0,1\}^n$ consisting of all vectors having Hamming weight~$k$. The slice appears naturally in combinatorics, coding theory, and elsewhere, and is known to algebraic combinatorialists as the \emph{Johnson scheme} $J(n,k)$. Every function on the slice can be represented uniquely as a multilinear polynomial $P$ in the $n$~input arguments $x_1,\ldots,x_n$ of degree at most $\min(k,n-k)$ which satisfies $\sum_{i=1}^n \frac{\partial P}{\partial x_i} = 0$ (the latter condition is known as \emph{harmonicity}). The \emph{degree} of a function on the slice is the degree of this polynomial. Alternatively, the degree of a function on the slice is the minimum degree of a polynomial in $x_1,\ldots,x_n$ (not necessarily multilinear or harmonic) which agrees with the function on all points of the slice.

A function $f$ on the slice is an \emph{$m$-junta} if there exist a function $g\colon\{0,1\}^m \to \{0,1\}$ and $m$~indices $i_1<\ldots<i_m$ such that $f(x) = g(x|_{i_1,\ldots,i_m})$, where $x|_{i_1,\ldots,i_m} = x_{i_1},\ldots,x_{i_m}$. Alternatively, $f$ is an $m$-junta if there exists a set~$I$ of $m$~coordinates such that $f$ is invariant under permutation of the coordinates in~$[n] \setminus I$; we also say that $f$ is an $I$-junta. Note that the set $I$ is not defined uniquely (in contrast to the hypercube case): for example, $f = \sum_{i \in I} x_i$ is both an $I$-junta and an $[n]\setminus I$-junta.

The \emph{$p$th norm} of $f$ is given by $\|f\|_p = \EE[|f|^p]^{1/p}$, where the expectation is over a uniform point in the slice. In particular, $\|f\|_2^2 = \EE[f^2]$.

Let $f$ be a Boolean function on the slice $\binom{[n]}{k}$, and let $P$ be its unique harmonic multilinear polynomial representation. The \emph{$d$th level} of $f$, denoted $f^{=d}$, is the homogeneous degree~$d$ part of~$P$ (the sum of all degree~$d$ monomials with their coefficients). The different levels are orthogonal: $\EE[f^{=d} f^{=e}] = 0$ if $d \neq e$. Orthogonality of the different levels implies that
\[
 \|f\|_2^2 = \sum_{d=0}^{\deg f} \|f^{=d}\|_2^2.
\]

Let $f$ be a Boolean function on the slice $\binom{[n]}{k}$, and let $i,j \in [n]$. We define $f^{(i\;j)}$ to be the function given by $f^{(i\;j)}(x) = f(x^{(i\;j)})$, where $x^{(i\;j)}$ is obtained from~$x$ by switching $x_i$ and $x_j$. The \emph{$(i,j)$th influence} of $f$ is $\Inf_{ij}[f] = \frac{1}{4} \Pr[f(x) \neq f(x^{(i\;j)})]$, where $x$ is chosen uniformly at random over the slice. An equivalent formula is $\Inf_{ij}[f] = \frac{1}{4} \EE[(f-f^{(i\;j)})^2]$. Clearly $\Inf_{ij}[f] = 0$ if and only if $f = f^{(i\;j)}$. The \emph{total influence} of $f$ is $\Inf[f] = \frac{1}{n} \sum_{1 \leq i < j \leq n} \Inf_{ij}[f]$. It is given by the formula
\begin{equation} \label{eq:influence}
 \Inf[f] = \sum_{d=0}^{\deg f} \frac{d(n+1-d)}{n} \|f^{=d}\|_2^2.
\end{equation}

 For a parameter $\rho \in (0,1]$, the \emph{noise operator} $T_\rho$, mapping functions on the slice to functions on the slice, is defined by
\begin{equation} \label{eq:noise}
 T_\rho f = \sum_{d=0}^{\deg f} \rho^{d(1-(d-1)/n)} f^{=d}.
\end{equation}
Alternatively, $(T_\rho f)(x)$ is the expected value of $f(y)$, where $y$ is chosen by applying $N \sim \mathrm{Po}(\frac{n-1}{2} \log (1/\rho))$ random transpositions to~$x$. Lee and Yau~\cite{LeeYau} proved a log~Sobolev inequality, which together with classical results of Diaconis and Saloff-Coste~\cite{DSC} implies that the following \emph{hypercontractive inequality} holds for some constant $C_H>0$:
\begin{equation} \label{eq:hypercontractivity}
 \|T_\rho f\|_2 \leq \|f\|_{4/3}, \qquad \rho = \left(\frac{2k(n-k)}{n(n-1)}\right)^{C_H}.
\end{equation}
For more information on functions on the slice, consult~\cite{fm}.

\section{Main theorem} \label{sec:main}

For the rest of this section, we fix an integer $d \geq 1$. Our goal is to prove Theorem~\ref{thm:main} for this value of~$d$.
We will use the phrase \emph{universal constant} to refer to a constant independent of $d$.

The strategy of the proof is to proceed in three steps:
\begin{enumerate}
 \item Bootstrapping: Every Boolean degree~$d$ function on $\binom{[2n]}{n}$ is a $K^d$-junta.
 \item Induction: If every Boolean degree~$d$ function on $\binom{[n]}{k}$ is an $M$-junta, then the same holds for $\binom{[n+1]}{k}$ and $\binom{[n+1]}{k+1}$ (under certain conditions).
 \item Culmination: If a Boolean degree~$d$ function on $\binom{[n]}{k}$ is an $L$-junta but not an $(L-1)$-junta, then (under certain conditions) there exists a Boolean degree~$d$ function on the hypercube depending on~$L$ coordinates.
\end{enumerate}

We also show a converse to the last step: given a Boolean degree~$d$ function on the hypercube depending on~$L$ coordinates, we show how to construct Boolean degree~$d$ functions on large enough slices that are $L$-juntas but not $(L-1)$-juntas.

\subsection{Bootstrapping} \label{sec:bootstrap}

We bootstrap our approach by proving that every Boolean degree~$d$ function on $\binom{[2n]}{n}$ is a junta. The proof is a simple application of hypercontractivity, and already appears in~\cite{fkmw}. We reproduce a simplified version here in order to make the paper self-contained.

The main idea behind the proof is to obtain a dichotomy on the influences of the function.

\begin{lemma} \label{lem:dichotomy}
 There exists a universal constant $\alpha$ such that all non-zero influences of a Boolean degree~$d$ function on $\binom{[2n]}{n}$ are at least $\alpha^d$.
\end{lemma}
\begin{proof}
 Let $f$ be a Boolean degree~$d$ function on $\binom{[2n]}{n}$. Given $i,j \in [n]$, consider the function $f_{ij} = (f - f^{(i\;j)})/2$, related to the $(i,j)$th influence of $f$ by $\Inf_{ij}[f] = \|f_{ij}\|_2^2$. Since $f$ is $0,1$-valued, $f_{ij}$ is $0,\pm 1$-valued. Since $f$ has degree~$d$, it follows that $f_{ij}$ can be written as a degree~$d$ polynomial, and so has degree at most~$d$. Hypercontractivity~\eqref{eq:hypercontractivity} implies that for some universal constant $\rho$, we have
\begin{equation} \label{eq:bootstrapping}
 \|T_\rho f_{ij}\|_2 \leq \|f_{ij}\|_{4/3}. 
\end{equation}
 We can estimate the left-hand side of~\eqref{eq:bootstrapping} using~\eqref{eq:noise}:
\[
 \|T_\rho f_{ij}\|_2^2 = \sum_{e=0}^d \rho^{2e(1-(e-1)/n)} \|f_{ij}^{=e}\|_2^2 \geq \rho^{2d} \sum_{e=0}^d \|f_{ij}^{=e}\|_2^2 = \rho^{2d} \|f_{ij}\|_2^2 = \rho^{2d} \Inf_{ij}[f].
\]
 We can calculate the right-hand side of~\eqref{eq:bootstrapping} using the fact that $f_{ij}$ is $0,\pm 1$-valued:
\[
 \|f_{ij}\|_{4/3}^{4/3} = \EE[|f_{ij}|^{4/3}] = \EE[|f_{ij}|^2] = \|f_{ij}\|_2^2 = \Inf_{ij}[f].
\]
 Combining the estimates on both sides of~\eqref{eq:bootstrapping}, we conclude that
\[
 \rho^{2d} \Inf_{ij}[f] \leq \Inf_{ij}[f]^{3/2}.
\]
 Hence either $\Inf_{ij}[f] = 0$ or $\Inf_{ij}[f] \geq \rho^{4d}$.
\end{proof}

The next step is to prove a degenerate triangle inequality (for the non-degenerate version, see Wimmer~\cite[Lemma 5.4]{Wimmer} and Filmus~\cite[Lemma 5.1]{filmus4}).

\begin{lemma} \label{lem:triangle-inequality}
 Let $f$ be a function on a slice. If $\Inf_{ik}[f] = \Inf_{jk}[f] = 0$ then $\Inf_{ij}[f] = 0$.
\end{lemma}
\begin{proof}
 If $\Inf_{ik}[f] = \Inf_{jk}[f] = 0$ then $f(x) = f(x^{(i\;k)}) = f(x^{(j\;k)})$, and so $f(x) = f(x^{(i\;k)(j\;k)(i\;k)}) = f(x^{(i\;j)})$. It follows that $\Inf_{ij}[f] = 0$.
\end{proof}

To complete the proof, we use formula~\eqref{eq:influence}, which implies that $\Inf[f] \leq d$ for any Boolean degree~$d$ function $f$, together with an idea of Wimmer~\cite[Proposition 5.3]{Wimmer}.

\begin{lemma} \label{lem:bootstrapping}
 There exists a universal constant $K > 1$ such that for $n \geq 2$, every Boolean degree~$d$ function on $\binom{[2n]}{n}$ is a $K^d$-junta.
\end{lemma}
\begin{proof}
 Let $f$ be a Boolean degree~$d$ function on $\binom{[2n]}{n}$. Construct a graph $G$ on the vertex set $[2n]$ by connecting two vertices $i,j$ if $\Inf_{ij}[f] \geq \alpha^d$, where $\alpha$ is the constant from Lemma~\ref{lem:dichotomy}. Let $M$ be a maximal matching in $G$. It is well-known that the $2|M|$ vertices of $M$ form a vertex cover $V$, that is, any edge of $G$ touches one of these vertices. Therefore if $i,j \notin V$ then $\Inf_{ij}[f] < \alpha^d$, and so $\Inf_{ij}[f] = 0$ according to Lemma~\ref{lem:dichotomy}. In other words, $f$ is a $V$-junta. It remains to bound the size of $V$.
 
 Let $(i,j)$ be any edge of $M$, and let~$k$ be any other vertex. Lemma~\ref{lem:triangle-inequality} shows that either $\Inf_{ik}[f] \neq 0$ or $\Inf_{jk}[f] \neq 0$, and so either $(i,k)$ or $(j,k)$ is an edge of $G$, according to Lemma~\ref{lem:dichotomy}. It follows that $G$ contains at least $|M|(2n-2)/2$ edges (we divided by two since some edges could be counted twice). Therefore $n\Inf[f] = \sum_{1 \leq i < j \leq n} \Inf_{ij}[f] \geq \alpha^d |M| (n-1)$. On the other hand,~\eqref{eq:influence} shows that
 \[
  \Inf[f] = \sum_{e=0}^d \frac{e(n+1-e)}{n} \|f^{=e}\|_2^2 \leq d \sum_{e=0}^d \|f^{=e}\|_2^2 = d \|f\|_2^2 \leq d. 
 \]
 It follows that
\[
 |V| = 2|M| \leq 2 \frac{n}{n-1} (1/\alpha)^d \leq (4/\alpha)^d. \qedhere
\]
\end{proof}

\subsection{Induction} \label{sec:induction}

The heart of the proof is an inductive argument which shows that if Theorem~\ref{thm:main} holds (with a non-optimal bound on the size of the junta) for the slice $\binom{[n]}{k}$, then it also holds for the slices $\binom{[n+1]}{k}$ and $\binom{[n+1]}{k+1}$, assuming that $n$ is large enough and that $k$ is not too close to~$0$ or~$n$. Given a Boolean degree~$d$ function $f$ on $\binom{[n+1]}{k}$ or $\binom{[n+1]}{k+1}$, the idea is to consider restrictions of $f$ obtained by fixing one of the coordinates.

\begin{lemma} \label{lem:restrictions}
 Suppose that every Boolean degree~$d$ function on $\binom{[n]}{k}$ is an $M$-junta, where $M \geq 1$.
 
 Let $f$ be a Boolean degree~$d$ function on $\binom{[n+1]}{k+b}$, where $b \in \{0,1\}$. For each $i \in [n+1]$, let $f_i$ be the restriction of $f$ to vectors satisfying $x_i = b$. Then for each $i \in [n+1]$ there exists a set $S_i \subseteq [n+1] \setminus \{i\}$ of size at most $M$ and a function $g_i\colon \{0,1\}^{S_i} \to \{0,1\}$, depending on all inputs, such that $f_i(x) = g_i(x|_{S_i})$.
\end{lemma}
\begin{proof}
 Choose $i \in [n+1]$. The domain of $f_i$ is isomorphic to $\binom{[n]}{k}$. Moreover, since $f$ can be represented as a polynomial of degree~$d$, so can $f_i$, hence $\deg f_i \leq d$. By assumption, $f_i$ is an $M$-junta, and so $f_i(x) = h_i(x|_{T_i})$ for some set $T_i$ of $M$ indices and some function $h_i\colon \{0,1\}^{T_i} \to \{0,1\}$. Let $S_i \subseteq T_i$ be the set of inputs that $h_i$ depends on. Then there exists a function $g_i\colon \{0,1\}^{S_i} \to \{0,1\}$ such that $g_i(x|_{S_i}) = h_i(x|_{T_i})$, completing the proof.
\end{proof}

Each of the sets $S_i$ individually contains at most $M$ indices. We now show that in fact they contain at most $M$ indices \emph{in total}.

\begin{lemma} \label{lem:union}
 Under the assumptions of Lemma~\ref{lem:restrictions}, suppose further that $n \geq (M+1)^3$ and $M+2 \leq k \leq n+1-(M+2)$.
 
 The union of any $M+1$ of the sets $S_i$ contains at most $M$ indices.
\end{lemma}
\begin{proof}
 We can assume, without loss of generality, that the sets in question are $S_1,\ldots,S_{M+1}$. Denote their union by $A$, and let $B = A \cup \{1,\ldots,M+1\}$. Since $|B| \leq (M+1)^2 \leq n$, there exists a point $r \in [n+1] \setminus B$. We proceed by bounding the number of unordered pairs of distinct indices $i,j \in [n+1] \setminus \{r\}$ such that $f_r = f_r^{(i\;j)}$, which we denote by $N$. Since $f_r$ is an $M$-junta, we know that $N \geq \binom{n-M}{2}$. We will now obtain an upper bound on $N$ in terms of $|A|$ and $|B|$.
 
 Let $1 \leq \ell \leq M+1$, and suppose that $i \in S_\ell$ and $j \notin S_\ell \cup \{\ell,r\}$. We claim that $f_r^{(i\;j)} \neq f_r$. Indeed, since $g_\ell$ depends on all inputs, there are two inputs $y,z$ to $g_\ell$, differing only on the $i$th coordinate, say $y_i = b$ and $z_i = 1-b$, such that $g_\ell(y) \neq g_\ell(z)$. Since $M+2 \leq k \leq n+1-(M+2)$, we can extend $y$ to an input $x$ to $f$ satisfying additionally the constraints $x_\ell = x_r = b$ and $x_j = 1-b$. Since $x_\ell = x_r = b$, the input $x$ is in the common domain of $f_\ell$ and $f_r$. Notice that $f_r(x) = f_\ell(x) = g_\ell(y)$, whereas $f_r(x^{(i\;j)}) = f_\ell(x^{(i\;j)}) = g_\ell(z)$, since $x_i = y_i = b$ whereas $x_j = 1-b$. By construction $g_\ell(y) \neq g_\ell(z)$, and so $f_r \neq f_r^{(i\;j)}$.
 
 The preceding argument shows that if $i \in A$ and $j \notin B \cup \{r\}$ then $f_r \neq f_r^{(i\;j)}$. Therefore $\binom{n}{2} - N \geq |A|(n-|B|) \geq |A|(n - (M+1)^2)$. Combining this with the lower bound $N \geq \binom{n-M}{2}$, we deduce that
\[
 |A| (n - (M+1)^2) \leq \binom{n}{2} - \binom{n-M}{2} = \frac{M(2n-M-1)}{2}.
\]
 Rearrangement shows that
\[
 |A| \leq \frac{M(2n-M-1)}{2(n-(M+1)^2)} =
 \left(1 + \frac{(M+1)(2M+1)}{2n-2(M+1)^2}\right)M.
\]
 When $n > (M^2 + (3/2)M + 1)(M+1)$, we have $\frac{(M+1)(2M+1)}{2n-2(M+1)^2} < \frac{1}{M}$, and so $|A| < M+1$. We conclude that when $n \geq (M+1)^3$, we have $|A| \leq M$.
\end{proof}

\begin{corollary} \label{cor:union}
 Under the assumptions of the preceding lemma, the union of $S_1,\ldots,S_{n+1}$ contains at most $M$ indices.
\end{corollary}
\begin{proof}
 Suppose that the union contained at least $M+1$ indices $i_1,\ldots,i_{M+1}$. Each index $i_t$ is contained in some set $S_{j_t}$, and in particular the union of $S_{j_1},\ldots,S_{j_{M+1}}$ contains at least $M+1$ indices, contradicting the lemma.
\end{proof}

Denoting the union of all $S_i$ by $S$, it remains to show that $f$ is an $S$-junta.

\begin{lemma} \label{lem:induction}
 Suppose that every Boolean degree~$d$ function on $\binom{[n]}{k}$ is an $M$-junta, where $M \geq 1$; that $n \geq (M+1)^3$; and that $M+2 \leq k \leq n+1-(M+2)$.
 
 Any Boolean degree~$d$ function on $\binom{[n+1]}{k}$ or on $\binom{[n+1]}{k+1}$ is an $M$-junta.
\end{lemma}
\begin{proof}
 Let $b,f_i,g_i,S_i$ be defined as in Lemma~\ref{lem:restrictions}, and let $S$ denote the union of $S_1,\ldots,S_{n+1}$. Corollary~\ref{cor:union} shows that $|S| \leq M$. Since $n \geq (M+1)^3$, it follows that there exists an index $r \in [n+1] \setminus S$. We will show that $f(x) = g_r(x|_{S_r})$, and so $f$ is an $M$-junta.
 
 Consider any input $x$ to $f$. If $x_r = b$ then $x$ is in the domain of $f_r$, and so clearly $f(x) = f_r(x) = g_r(x|_{S_r})$. Suppose therefore that $x_r = 1-b$. Since $M+2 \leq k \leq n+1-(M+2)$, there exists a coordinate $s \in [n+1] \setminus S$ such that $x_s = b$, putting $x$ in the domain of $f_s$. Again since $M+2 \leq k \leq n+1-(M+2)$, there exists a coordinate $t \in [n+1] \setminus (S \cup \{s\})$ such that $x_t = b$. Since $x^{(r\;t)}$ is in the domain of $f_r$, we have
\[
 f(x) = f_s(x) = g_s(x|_{S_s}) = g_s(x^{(r\;t)}_{S_s}) = f_s(x^{(r\;t)}) = f_r(x^{(r\;t)}) = g_r(x^{(r\;t)}|_{S_r}) = g_r(x|_{S_r}). \qedhere
\]
\end{proof}

\subsection{Culmination} \label{sec:culmination}

Combining Lemma~\ref{lem:bootstrapping} and Lemma~\ref{lem:induction}, we obtain a version of Theorem~\ref{thm:main} with a suboptimal upper bound on the size of the junta.

\begin{lemma} \label{lem:main}
 There exists a universal constant $C > 1$ such that whenever $C^d \leq k \leq n - C^d$, every Boolean degree~$d$ function on $\binom{[n]}{k}$ is a $C^d$-junta.
\end{lemma}
\begin{proof}
 Let $K$ be the constant from Lemma~\ref{lem:bootstrapping}, and let $M = K^d$. We choose $C := (K+2)^3$.
 
 Let us assume that $k \leq n/2$ (the proof for $k \geq n/2$ is very similar). Lemma~\ref{lem:bootstrapping} shows that every Boolean degree~$d$ function on $\binom{[2k]}{k}$ is an $M$-junta. If $m \geq 2k$ then $m \geq 2k \geq (M+1)^3$ and $M+2 \leq k \leq m-(M+2)$. Therefore Lemma~\ref{lem:induction} shows that if every Boolean degree~$d$ function on $\binom{[m]}{k}$ is an $M$-junta, then the same holds for $\binom{[m+1]}{k}$. Applying the lemma $n-2k$ times, we conclude that every Boolean degree~$d$ function on $\binom{[n]}{k}$ is a $C^d$-junta.
\end{proof}

To complete the proof of the theorem, we show how to convert a Boolean degree~$d$ function on the slice depending on many coordinates to a Boolean degree~$d$ function on the hypercube depending on the same number of coordinates.

\begin{lemma} \label{lem:slice-to-cube}
 Suppose that $f$ is a Boolean degree~$d$ function on $\binom{[n]}{k}$ which is an $L$-junta but not an $(L-1)$-junta, where $k \leq L \leq n-k$. Then there exists a Boolean degree~$d$ function $g$ on $\{0,1\}^L$ which depends on all coordinates.
\end{lemma}
\begin{proof}
 Without loss of generality, we can assume that $f(x) = g(x_1,\ldots,x_L)$ for some Boolean function $g$ on the $L$-dimensional hypercube. Since $f$ is not an $(L-1)$-junta, the function $g$ depends on all coordinates. Since $k \leq L \leq n-k$, as $x$ goes over all points in $\binom{[n]}{k}$, the vector $x_1,\ldots,x_L$ goes over all points in $\{0,1\}^L$. It remains to show that there is a degree~$d$ polynomial agreeing with~$g$ on $\{0,1\}^L$.
 
 Since $f$ has degree at most~$d$, there is a degree~$d$ multilinear polynomial $P$ such that $f = P$ for every point in $\binom{[n]}{k}$. If $\pi$ is any permutation of $\{L+1,\ldots,n\}$ then $f = f^\pi$, where $f^\pi(x) = f(x^\pi)$. Denoting the set of all such permutations by $\Pi$, if we define $Q := \EE_{\pi \in \Pi}[P^\pi]$ then $f = Q$ for every point in $\binom{[n]}{k}$. The polynomial $Q$ is a degree~$d$ polynomial which is invariant under permutations from $\Pi$. For $a_1,\ldots,a_L \in \{0,1\}^L$ summing to $a \leq d$, let $Q_{a_1,\ldots,a_L}(x_{L+1},\ldots,x_n) = Q(a_1,\ldots,a_L,x_{L+1},\ldots,x_n)$. This is a degree~$d-a$ symmetric polynomial, and so a classical result of Minsky and Papert~\cite{Minsky:1988:PEE:50066} (see also~\cite[Lemma~3.2]{NisanSzegedy}) implies that there exists a degree~$d-a$ univariate polynomial $R_{a_1,\ldots,a_L}$ such that $Q_{a_1,\ldots,a_L}(x_{L+1},\ldots,x_n) = R_{a_1,\ldots,a_L}(x_{L+1} + \cdots + x_n)$ for all $x_{L+1},\ldots,x_n \in \{0,1\}^{n-L}$. Since $x_{L+1} + \cdots + x_n = k - x_1 - \cdots - x_L$, it follows that for inputs $x$ in $\binom{[n]}{k}$, we have
\[
 f(x_1,\ldots,x_n) = \sum_{\substack{a_1,\ldots,a_L \in \{0,1\}^L \\ a_1+\cdots+a_L \leq d}} \prod_{i\colon a_i=1} x_i \cdot R_{a_1,\ldots,a_L}(k - x_1 - \cdots - x_L).
\]
 The right-hand side is a degree~$d$ polynomial in $x_1,\ldots,x_L$ which agrees with $g$ on $\{0,1\}^L$.
\end{proof}

Theorem~\ref{thm:main} immediately follows from combining Lemma~\ref{lem:main} and Lemma~\ref{lem:slice-to-cube}.

\smallskip

We conclude this section by proving a converse of Lemma~\ref{lem:cube-to-slice}.

\begin{lemma} \label{lem:cube-to-slice}
 Suppose that $g$ is a Boolean degree~$d$ function on $\{0,1\}^L$ depending on all coordinates. Then for all $n,k$ satisfying $L \leq k \leq n-L$ there exists a Boolean degree~$d$ function $f$ on $\binom{[n]}{k}$ which is an $L$-junta but not an $(L-1)$-junta.
\end{lemma}
\begin{proof}
 We define $f(x_1,\ldots,x_n) = g(x_1,\ldots,x_L)$. Clearly, $f$ is an $L$-junta. Since $g$ has degree at most~$d$, there is a polynomial $P$ which agrees with $g$ on all points of $\{0,1\}^L$. The same polynomial also agrees with $f$ on all points of $\binom{[n]}{k}$, and so $f$ also has degree at most~$d$. It remains to show that $f$ is not an $(L-1)$-junta.
 
 Suppose, for the sake of contradiction, that $f$ were an $(L-1)$-junta. Then there exists a set $S$ of size at most $L-1$ and a Boolean function $h\colon \{0,1\}^S \to \{0,1\}$ such that $f(x) = h(x|_S)$. Since $|S| < L$, there exists some coordinate $i \in \{1,\ldots,L\} \setminus S$. Since $g$ depends on all coordinates, there are two inputs $y,z$ to $g$ differing only in the $i$th coordinate, say $y_i = 0$ and $z_i = 1$, such that $g(y) \neq g(z)$. Since $n \geq 2L$, there exists a coordinate $j \in [n] \setminus (\{1,\ldots,L\} \cup S)$. Since $L \leq k \leq n-L$, we can extend $y$ to an input $\tilde{y}$ to $f$ such that $x_j = 1$. The input $\tilde{z} = \tilde{y}^{(i\;j)}$ extends $z$. Since $i,j \notin S$, the inputs $\tilde{y},\tilde{z}$ agree on all coordinates in $S$, and so $f(\tilde{y}) = h(\tilde{y}|_S) = h(\tilde{z}|_S) = f(\tilde{z})$. On the other hand, $f(\tilde{y}) = g(y) \neq g(z) = f(\tilde{z})$. This contradiction shows that $f$ cannot be an $(L-1)$-junta.
\end{proof}

\section{Discussion} \label{sec:discussion}

\paragraph{Optimality.}
Lemma~\ref{lem:cube-to-slice} shows that the size of the junta in Theorem~\ref{thm:main} is optimal. However, it is not clear whether the bounds on $k$ are optimal. The theorem fails when $k \leq d$ or $k \geq n-d$, since in these cases every function has degree~$d$. This prompts us to define the following two related quantities:
\begin{enumerate}
 \item $\zeta(d)$ is the minimal value such that every Boolean degree~$d$ function on $\binom{[n]}{k}$ is an $O(1)$-junta whenever $\zeta(d) \leq k \leq n-\zeta(d)$.
 \item $\xi(d)$ is the minimal value such that every Boolean degree~$d$ function on $\binom{[n]}{k}$ is a $\gamma(d)$-junta whenever $\xi(d) \leq k \leq n-\xi(d)$.
\end{enumerate}
Clearly $d < \zeta(d) \leq \xi(d)$. We can improve this to $\zeta(d) \geq \eta(d) \geq 2 \lceil \frac{d+1}{2} \rceil$, where $\eta(d)$ is the maximum integer such that there exists a non-constant univariate degree~$d$ polynomial $P_d$ satisfying $P_d(0),\ldots,P_d(\eta(d)-1) \in \{0,1\}$. Given such a polynomial $P_d$, we can construct a degree~$d$ function $f_d$ on $\binom{[n]}{k}$ which is not a junta:
\[
 f_d(x_1,\ldots,x_n) = P_d(x_1+\cdots+x_{\lfloor n/2 \rfloor}).
\]
When $k < \eta(d)$, the possible values of $x_1+\cdots+x_{\lfloor n/2 \rfloor}$ are such that $f_d$ is Boolean. One can check that $f_d$ is not an $L$-junta unless $L \geq \lfloor n/2 \rfloor$. The following polynomial shows that $\eta(d) \geq 2\lceil \frac{d+1}{2} \rceil$:
\[
 P_d(\sigma) = \sum_{e=0}^d (-1)^e \binom{\sigma}{e},
\]
where $\binom{\sigma}{0} = 1$. While this construction can be improved for specific~$d$ (for example, $\eta(7) = 9$ and $\eta(12) = 16$), the upper bound $\eta(d) \leq 2d$ shows that this kind of construction cannot given an exponential lower bound on $\zeta(d)$.

Curiously, essentially the same function appears in~\cite[Section~7]{dfh2} as an example of a degree~$d$ function on the biased hypercube which is almost Boolean but somewhat far from being constant.

\paragraph{Extensions.}
It would be interesting to extend Theorem~\ref{thm:main} to other domains. In recent work~\cite{FilmusIhringer1}, we explored Boolean degree~$1$ functions on various domains, including various association schemes and finite groups, and the multislice (consult the work for the appropriate definitions). Inspired by these results, we make the following conjectures:
\begin{enumerate}
 \item If $f$ is a Boolean degree~$d$ function on the symmetric group then there are sets $I,J$ of $O(1)$ indices such that $f(\pi)$ depends only on $\pi(i)|_{i \in I}$ and $\pi^{-1}(j)|_{j \in J}$.
 \item If $f$ is a Boolean degree~$d$ function on the Grassmann scheme then there are $O(1)$ points and hyperplanes such that $f(S)$ depends only on which of the points is contained in $S$, and which of the hyperplanes contain $S$.
 \item If $f$ is a Boolean degree~$d$ function on the multislice $M(k_1,\ldots,k_m)$ for $k_1,\ldots,k_m \geq \exp(d)$ then $f$ is a $\gamma_m(d)$-junta, where $\gamma_m(d)$ is the maximum number of coordinates that a Boolean degree~$d$ function on the Hamming scheme $H(n,m)$ can depend on.
\end{enumerate}
We leave it to the reader to show that $\gamma_m(d)$ exists. In fact, simple arguments show that $m^{d-1} \leq \gamma_m(d) \leq \gamma(\lceil \log_2 m \rceil d)$.

\bibliographystyle{plain}
\bibliography{biblio_plus}

\end{document}